\documentclass[12pt,draftcls,onecolumn]{IEEEtran}  

\usepackage{dsfont}
\usepackage{graphicx}
\usepackage{epsfig}
\usepackage{subfig}
\usepackage{times}
\usepackage{amsmath}
\usepackage{amssymb}
\usepackage{cite}
\usepackage{color}
\usepackage{arydshln}
\usepackage[left, pagewise]{lineno}

\newtheorem{Problem}{Problem}[section]

\newtheorem{Remark}{Remark}[section]
\newtheorem{Lemma}{Lemma}[section]
\newtheorem{Theorem}{Theorem}[section]

\newtheorem{Assumption}{Assumption}[section]

\graphicspath{{figure/}}

\title{Formation Control of Rigid Graphs with a Flex Node Addition}
\author{Viet Hoang Pham$^{\dag}$, Minh Hoang Trinh$^{\dag}$, and Hyo-Sung Ahn$^{\dag}$ 
\thanks{The full version of this paper with general extensions has been submitted to a journal for publication.}
\thanks{\small $^{\dag}$The authors are with School of Mechanical Engineering,
Gwangju Institute of Science and Technology, Gwangju, Korea.
E-mail: {vietph@gist.ac.kr}; {trinhhoangminh@gist.ac.kr}; {hyosung@gist.ac.kr.}}%
}

\begin{document}

\maketitle 
\thispagestyle{empty}
\pagestyle{empty}


\begin{abstract}
This paper examines stability properties of distance-based formation control when the underlying topology consists of a rigid graph and a flex node addition. It is shown that the desired equilibrium set is locally asymptotically stable but there exist undesired equilibria. Specifically, we further consider two cases where the rigid graph is a triangle in $2$-D and a tetrahedral in $3$-D, and prove that any undesired equilibrium point in these cases is unstable. Thus in these cases, the desired formations are almost globally asymptotically stable.
\end{abstract}

\section{Introduction}
As a solution of the distance-based formation control problem, gradient descent control laws have been extensively studied \cite{Kwangkyo2015,Krick2009,Dorler2010,Dimarogonas2010,Dasgupta2011,Kwangkyo2014,Park2014,Sun2015}. Given a system of $n$ single integrator agents employing the gradient control law derived from some potential functions, it is well known that local asymptotic stability of the formations is guaranteed when the interaction graph is undirected and rigid \cite{Krick2009,Kwangkyo2014}.

Several results on (almost) global stability of these formations can also be found, for examples, the three-agent formation in the $2$-D space \cite{Dorler2010} or the four-agent formation in $3$-D space \cite{Park2014}. A common strategy adopted in these papers is showing non-existence or instability of the undesired equilibrium set. Then, if the formation is initially in a generic position \cite{Sun2015} and is not in the undesired equilibrium set, it will asymptotically converge to a point in the desired equilibrium set, i.e., a desired formation. The existence of undesired equilibria associated with a undesired formation shape can be observed from numerical simulations when $n \ge 4$ \cite{Krick2009,Dasgupta2011}. In fact, the gradient descent control laws fail to globally stabilize $n$-agents formations.

Alternative control laws were proposed, e.g., designing control weights for stabilization of affine formations \cite{Zhiyun2016}; simultaneously aligning the agents' local coordinate frames and controlling the relative position \cite{Miguel2015,Kwangkyo20142}; or perturbing the agents' trajectories by quasi-random directional noises to escape unstable undesired equilibria \cite{Tian2013}. These strategies provide global convergence of the formation to the desired shape,  however, there are also trade-off on these solutions. In affine formations, all agents are required to have the same coordinate systems. Orientation alignment algorithm requires exchanging information between agents. The perturbations cannot guarantee a global convergence to the desired formation since stable undesired equilibria could exist.

In almost all works have been reported \cite{Kwangkyo2015,Krick2009,Dorler2010,Dimarogonas2010,Dasgupta2011,Kwangkyo2014,Park2014,Sun2015}, the desired formation graphs are usually assumed to be
rigid since these graphs preserve the formation shape at least in local sense. However, there are scenarios in which we do not need all agents in the system to be remained in a rigid shape. For example, consider a group consisting of several vehicles which have to move in a prescribed formation in the plane and a flying UAV whose partial tasks are supervising or guiding these vehicles to a desired region. Practically, the UAV only needs to keep a distance constraint to a vehicle and saves its remaining degree of freedoms for other tasks. This paper devotes to study these scenarios. More specifically, we examine the distance based formation problem when the underlying graph is a rigid graph adding a flex node. Two specific cases are studied in detail in this paper, in which the rigid graphs are the triangle in $2$-D space and the tetrahedral in $3$-D space.

Consequently, the contributions of this paper can be summarized as follows. We give analysis on the effect of the added flex node to the rigid formation. Although the flex node may act as a disturbance to the rigid formation, the set of all desired distance constraints specified in the overall formation is proven to be locally asymptotically stable. Furthermore, when the rigid graph is triangle in the plane or a tetrahedral in $3$-D space, any undesired equilibrium point is unstable, which implies the desired formations are almost globally asymptotically stable. To examine the effects of motions of flex node more rigorously, we further suppose that the flex node is governed by a finite velocity or required to go to a specific position. Under these circumstances, it is still shown that the desired formations are almost globally asymptotically stable. 


The rest of this paper is organized as follows. In Section II, we briefly review some background related with formation control. In Section III, we show instability of undesired equilibrium points for two specific cases: a triangle adding a flex node in $2$-D space and a tetrahedron adding a flex node in $3$-D space. In Section IV, we consider the case when the flex agent has one more additional control input to go to a specific position. Simulations supporting our analysis are provided in Section V. Finally, Section VI provides the concluding remarks.

\section{Preliminaries}
\subsection{Graph representation of the formation}
We use an undirected graph $\mathcal{G} = \left( {\mathcal{V}, \mathcal{E}} \right)$ to describe the underlying topology of agents. Each agent corresponds to a node of the graph, and each edge linking nodes $i$ and $j$ determines a distance constraint that needs to be preserved.
The edge set $\mathcal{E}$ can be partitioned as $\mathcal{E} = \mathcal{E}_+ \cup \mathcal{E}_-$ such that $\mathcal{E}_+ \cap \mathcal{E}_- = \emptyset$ and $(i,j) \in \mathcal{E}_+$ if and only if $(j,i) \in \mathcal{E}_-$. For simplicity, we assume $\mathcal{E}_+ = \left \{ (i,j) \in \mathcal{E}: i<j \right \}$ and use $\epsilon_k$ to denote an edge in $\mathcal{E}_+$, $\mathcal{E}_+ = \left \{ \epsilon_1,...,\epsilon_m \right \}$. Denote $\textbf{B} = {\left[ {{b_{ij}}} \right]_{n \times m}}$ as the incidence matrix of graph
\[{b_{ij}} = \left\{ \begin{array}{lcl}
1&\textrm{ if i is the sink node of edge } \epsilon_j, \\
-1&\textrm{ if i is the source node of edge } \epsilon_j, \\
0&\textrm{otherwise.}
\end{array} \right.\]
With Kronecker product, we define matrix $\bar {\textbf{B}} = \textbf{B} \otimes {\textbf{I}_d}$, where $\textbf{I}_d$ is $d \times d$ identify matrix.
Let ${\textbf{p}_i} \in \mathbb{R} {^d}$ be the position  of agent $i$. The stacked vector $\textbf{p} = \left[ \textbf{p}_1^T, \textbf{p}_2^T,..., \textbf{p}_{N+1}^T \right]^T \in {\mathbb{R}^{\left( N+1 \right)d}}$ represents the realization of graph $\mathcal{G}$. Without notation confusion, for $k$-th edge in $\mathcal{E}_+$ linking nodes $i$ and $j$, $\epsilon_k=(i,j) \in \mathcal{E}_+$, we denote the relative position vector $\textbf{z}_{\epsilon_k} = \textbf{z}_{ij} = \textbf{p}_i-\textbf{p}_j$. Let $\textbf{z} = \left[ \textbf{z}_{\epsilon_1}^T, \textbf{z}_{\epsilon_2}^T,..., \textbf{z}_{\epsilon_m}^T \right]^T$, then we have
\[\textbf{z} = \bar {\textbf{B}}^T \textbf{p}.\]
Assume $\mathcal{G}_r=\left( \mathcal{V}_{\mathcal{G}_r},\mathcal{E}_{\mathcal{G}_r} \right)$ is a rigid graph where $\mathcal{V}_{\mathcal{G}_r}=\left\{ {1,2,...,N} \right\}$. In this paper, we consider formations whose underlying graph contains rigid graph $\mathcal{G}_r$ and an additional flex node, i.e., $\mathcal{G} = \left( {\mathcal{V}, \mathcal{E}} \right)$, $\mathcal{V} = \left\{ {1,2,...,N,N+1} \right\}$ and $\mathcal{E}=\mathcal{E}_{\mathcal{G}_r} \cup \left\{(N,N+1)\right\}$.
\begin{figure}[htb]
\begin{center}
\includegraphics[width = 2 in]{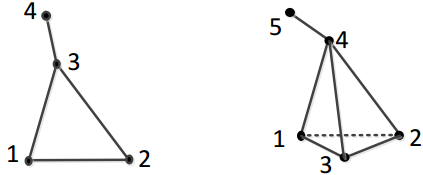}
\end{center}
\caption{Two examples of rigid graph adding a flex node: A triangle (1,2,3) adding a flex node (4) moving in the plane and a tetrahedron (1,2,3,4) adding a flex node (5) moving in $3$-D space.}
\label{fig:topologygraph}
\end{figure}
\subsection{Distance-based formation control problem}
We consider a group of autonomous mobile agents moving in $d$-dimensional Euclidean space ($d=2,3$). Assume that each agent obeys a single integrator dynamics of the form: 
\begin{equation}
{{\dot {\textbf{p}}}_i} = {\textbf{u}_i},
\label{eq:system}
\end{equation}
where $\textbf{u}_i\in {\mathbb{R}^d}$ is agent $i$'s control input. Let $\mathcal{D} = \left \{ \bar d_{ij}: (i,j) \in \mathcal{E} \right \}$ be the set of desired distances between neighboring agents and assume that $\mathcal{D}$ is feasible, which means if $(i,j),(j,k),(k,i) \in \mathcal{E}$ then
\begin{equation}
d_{ij}+d_{jk}>d_{ki},d_{jk}+d_{ki}>d_{ij},d_{ki}+d_{ij}>d_{jk}
\label{eq:relizable_ineq}
\end{equation} 
for all $i,j,k \in \mathcal{V}$. Define the desired formation set as
\begin{equation}
\mathcal{Q}_C = \left \{ \textbf{p} \in \mathbb{R}^{(N+1)d}: \dot{\textbf{p}} = 0 ; ||\textbf{z}_{ij}||= \bar d_{ij}, \forall (i,j) \in \mathcal{E} \right \}
\label{eq:desired_formation}
\end{equation} 
where $||.||$ is the Euclidean norm.
Denote the set of neighbours of agent $i$ by $\mathcal{N}_i$ and assume that each agent $i$ can only measure the relative position of its neighbours in its own coordinate system, $\textbf{p}_j^i, j \in \mathcal{N}_i$.
The main task of distance-based formation control can be summarized as follow: 
\begin{Problem}
For a given system of single-integrator modelled agents (\ref{eq:system}) moving in the $d$-dimensional space ($d=2,3$), design a distributed control law, for which each agent $i$ uses only distance measurement $\textbf{p}_j^i, j \in \mathcal{N}_i$, such that the formation of the system converges to a desired formation in $\mathcal{Q}_C$.
\label{p:fc_si}
\end{Problem}
\subsection{The gradient control law}
Let $\phi(x,\bar d)$ be a function that satisfies the following assumption:
\begin{Assumption}
If $\bar d$ is a constant, then
\begin{itemize}
\item $\phi(x,\bar d)$ is non-negative and $g(x,\bar d):= \frac{\partial \phi(x,\bar d)}{\partial x}$ is strictly monotonically increasing,
\item $\phi(x,\bar d)$, $g(x,\bar d)$ are continuously differentiable on $x \in (-{\bar d}^2, \infty)$ and equal zero if and only if $x=0$,
\item $\phi(x,\bar d)$ is analytic in a neighbourhood of 0.
\end{itemize}
\label{as:potentialfunction}
\end{Assumption}
Let $e_{\epsilon_k}=e_{ij}=||\textbf{z}_{ij}||^2-\bar d_{ij}^2$ be the squared distance error for edge $\epsilon_k = (i,j) \in \mathcal{E}$, and $g_{\epsilon_k} = g_{ij} = g(e_{ij},\bar{d}_{ij})$.
Let us define a local potential for each agent $V_i: \mathbb{R}^{d(|\mathcal{N}_i|+1)} \to \mathbb{R}_+$
\begin{equation}
V_i({\textbf{p}_i^i},...,{\textbf{p}_j^i},...) = \frac{1}{2} \sum\limits_{j \in \mathcal{N}_i} {\phi (e_{ij}, \bar{d}_{ij})} 
\label{eq:potential_function_local}
\end{equation}
and a global potential function for system as
\begin{equation}
V = \sum\limits_{i \in \mathcal{V}} {V_i} = \frac{1}{2} \sum\limits_{i \in \mathcal{V}} {\sum\limits_{j \in \mathcal{N}_i} {\phi (e_{ij}, \bar{d}_{ij})}}
\label{eq:potential_function_global}
\end{equation} 
From the potential function\footnote{There are a lot of potential functions satisfying Assumption \ref{as:potentialfunction} which are widely used in the literature. 
For example, they are $\phi(e_{\epsilon_k},\bar{d}_{\epsilon_k})=e_{\epsilon_k}^2$ and $g_{\epsilon_k} = e_{\epsilon_k}$ used in  \cite{Krick2009,Kwangkyo2014}, and $\phi(e_{\epsilon_k},\bar{d}_{\epsilon_k})=\frac{e_{\epsilon_k}^2}{{\epsilon_k}+\bar{d}_{\epsilon_k}^2}$ and $g_{\epsilon_k} = 1-\frac{d_{\epsilon_k}^4}{{(\epsilon_k}+\bar{d}_{\epsilon_k}^2)^2}$ used in \cite{Dimarogonas2010,Tian2013}.}, we can define the gradient-descent control law for agents:
\begin{equation}
\dot {\textbf{p}} = \textbf{u} = - \nabla_{\textbf{p}} V(\textbf{p})
\label{eq:gradient_controllaw_global}
\end{equation}
where $\textbf{u} = {\left[ {\begin{array}{*{20}{c}}{\textbf{u}_1^T}&{\textbf{u}_2^T}&{...}&{\textbf{u}_N^T}&{\textbf{u}_{N+1}^T}\end{array}} \right]^T} \in {\mathbb{R}^{\left( N+1 \right)d}}$ is the control input vector.
The detailed control law for each agent:
\begin{equation}
\dot {\textbf{p}}_i = \textbf{u}_i = - \nabla_{\textbf{p}_i} V(\textbf{p}) = - \sum\limits_{j \in \mathcal{N}_i} {g_{ij} \textbf{z}_{ij}} 
\label{eq:gradient_controllaw_local}
\end{equation}
\begin{Remark}
We can see $||\mathbf{z}_{ij}||=||\mathbf{p}_j^i||$. Let $\mathbf{Q}^i$ be the rotation matrix of global coordinate system with respect to agent $i$'s coordinate system. Then the control input of agent $i$ in its own coordinate system is given as $\textbf{u}_i^i = \mathbf{Q}^i \textbf{u}_i = -g_{ij} \mathbf{Q}^i \mathbf{z}_{ij}=-\sum\limits_{j \in \mathcal{N}_i}{g_{ij}\mathbf{p}_j^i}$. So, the control law (\ref{eq:gradient_controllaw_local}) does not require that agents' coordinate systems are aligned. 
\end{Remark}

Denote the equilibrium set
\begin{equation}
\mathcal{Q} = \{ {\textbf{p}: \dot{\textbf{p}} = 0 \textrm{ and } \sum\limits_{j \in \mathcal{N}_i} {g_{ij} \textbf{z}_{ij} = 0}, \forall i \in \mathcal{V}} \}.
\label{eq:EqSet}
\end{equation} 
By the similar proof in \cite{Kwangkyo2014}, we have
\begin{Theorem}
For a given system (\ref{eq:system}) with the interaction graph $\cal{G}$, under control law (\ref{eq:gradient_controllaw_global})
\begin{itemize}
\item  $p(t)$ approaches $\mathcal{Q}$ as $t \to \infty $.
\item  The desired formation $\mathcal{Q}_C$ is locally asymptotically stable.
\end{itemize}
\label{theorem:EqSet}
\end{Theorem}
\begin{Remark}
Under the control law (\ref{eq:gradient_controllaw_global}), $\dot{\textbf{p}} = 0$ implies $\sum\limits_{j \in \mathcal{N}_i} {g_{ij} \textbf{z}_{ij} = 0, \forall i \in \mathcal{V}}$ and $\mathcal{Q}_C = \left\{ {\textbf{p}:{g_{ij} = 0}, \forall (i,j) \in \mathcal{E}} \right\}$.\\
The control law (\ref{eq:gradient_controllaw_global}) does not guarantee the global convergence to desired formation. There also exist undesired equilibria as 
\begin{equation}
{\mathcal{Q}_I} =\mathcal{Q} \setminus \mathcal{Q}_C = \left\{ {\textbf{p} \in Q: \exists (i,j) \in \mathcal{E}, g_{ij} \neq 0 } \right\}
\end{equation} 
\end{Remark} 
\section{Stability analysis of undesired equilibrium}
\subsection{Hessian of potential function}
The Jacobian $\textbf{J}(\textbf{p})$ of the right-hand side of (\ref{eq:gradient_controllaw_global}) is same as the negative Hessian of the potential function $V$. Let $\textbf{p}^* \in \mathcal{Q}_I$ be an undesired equilibrium point; then $\textbf{p}^*$ is stable if and only if all eigenvalues of $\textbf{J}(\textbf{p}^*)$ are not positive or all eigenvalues of $\textbf{H}_{\textbf{V}}(\textbf{p}^*)$ are not negative, i.e. the Hessian $\textbf{H}_{\textbf{V}}(\textbf{p}^*)$ is positive semidefinite.  
We have 
\[ \frac{\partial V}{\partial {\textbf{p}}^T_j} = \frac{1}{2}\sum\limits_{k=1}^{m} {\left( \frac{\partial V}{\partial e_{\epsilon_k}} \frac{\partial e_{\epsilon_k}}{\partial {\textbf{p}}_j} \right) ^T } = \sum\limits_{k=1}^{m} {{ g_{\epsilon_k} {\textbf{z}}_{\epsilon_k}^T } \frac{\partial {\textbf{z}}_{\epsilon_k}}{\partial {\textbf{p}}_j}}\]
From definition of ${\textbf{z}}_{\epsilon_k}$, we can see $\frac{\partial^2 {\textbf{z}}_{\epsilon_k}}{\partial {\textbf{p}}_i \partial {\textbf{p}}^T_j}=0 \textrm{ } \forall i,j \in \mathcal{V}$; so
\begin{align}
 \frac{\partial^2 V}{\partial {\textbf{p}}_i \partial {\textbf{p}}^T_j} &=  \sum\limits_{k=1}^{m} { \left( \frac{\partial g_{\epsilon_k}}{\partial e_{\epsilon_k}} \frac{\partial e_{\epsilon_k}}{\partial {\textbf{p}}_i} {\textbf{z}}_{\epsilon_k}^T + g_{\epsilon_k} \frac{\partial {\textbf{z}}_{\epsilon_k}^T}{\partial {\textbf{p}}_i} \right)} \frac{\partial {\textbf{z}}_{\epsilon_k}}{\partial {\textbf{p}}_j} \nonumber\\
  &= \sum\limits_{k=1}^{m} { \frac{\partial {\textbf{z}}_{\epsilon_k}^T}{\partial {\textbf{p}}_i} \left(2 \frac{ \partial g_{\epsilon_k}}{\partial e_{\epsilon_k}} {\textbf{z}}_{\epsilon_k} {\textbf{z}}_{\epsilon_k}^T + g_{\epsilon_k} \otimes \textbf{I}_d \right) \frac{\partial {\textbf{z}}_{\epsilon_k}}{\partial {\textbf{p}}_j}}
\end{align}
which can be further compactly written as 
\begin{equation}
\textbf{H}_{\textbf{V}} = \bar{\textbf{B}} \textbf{M} \bar{\textbf{B}}^T.
\label{eq:HessianMatrix}
\end{equation}
with $\textbf{M}:=diag \left( 2 \frac{ \partial g_{\epsilon_k}}{\partial e_{\epsilon_k}} \textbf{z}_{\epsilon_k} \textbf{z}_{\epsilon_k}^T + g_{\epsilon_k} \otimes \textbf{I}_d \right) \in \mathbb{R}^{m \times m}$. We can see that the sum of elements in one column or one row is zero.
Next, we examine the stability of undesired equilibria in two specific cases: 
\begin{itemize}
\item A triangle adding a flex node in the plane.
\item A tetrahedron adding a flex node in $3$-D space.
\end{itemize}
For convenience, we define $\textbf{E} = \textbf{B} diag \left( g_{\epsilon_k} \right) \textbf{B}^T$, $\rho_{\epsilon_k} = \frac{ \partial g_{\epsilon_k}}{\partial e_{\epsilon_k}}$ and $\textbf{R}_{[i]} = \textbf{B} [\textbf{z}_{\epsilon_1[i]} \sqrt{\rho_{\epsilon_1}},...,\textbf{z}_{\epsilon_m[i]} \sqrt{\rho_{\epsilon_m}}]^T, [i]=x,y,z$. From Assumption \ref{as:potentialfunction}, $\rho_{\epsilon_k}>0$, and $g_{\epsilon_k}>0$ (respectively $<,=$) if $e_{\epsilon_k}>0$ (respectively $<,=$) or $||\textbf{z}_{\epsilon_k}||>\bar{d}_{\epsilon_k}$ (respectively $<,=$). 
\subsection{Triangle adding a flex node in the plane}
We use a column-reordering transformation $\textbf{T}$ such that $\textbf{R} \textbf{T} = \left[ {\begin{array}{*{20}{c}}{{\textbf{R}_x}}&{{\textbf{R}_y}} \end{array}} \right]$. The transformed Hessian matrix is given by
\begin{equation}
\textbf{H} = {\textbf{T}^T}{\textbf{H}}_{\textbf{V}}\textbf{T} = \left[ {\begin{array}{*{20}{c}}
{2\textbf{R}_x^T{\textbf{R}_x} + \textbf{E}}&{2\textbf{R}_x^T{\textbf{R}_y}}\\
{2\textbf{R}_y^T{\textbf{R}_x}}&{2\textbf{R}_y^T{\textbf{R}_y} + \textbf{E}}
\end{array}} \right]
\label{eq:Hessian2D}
\end{equation}
Since $\textbf{T}$ is orthogonal, i.e., $\textbf{T}^T=\textbf{T}^{-1}$, the eigenvalues of $\textbf{H}$ and $\textbf{H}_\textbf{V}$ are same. Denote $\textbf{H}_{22} (\textbf{p}^*)=2\textbf{R}_y^T(\textbf{p}^*){\textbf{R}_y}(\textbf{p}^*) + \textbf{E}(\textbf{p}^*)$ where $\textbf{p}^*$ is an undesired equilibrium. Consider the vector $\textbf{u}=\left[ {\begin{array}{*{20}{c}}{{\textbf{0}}}&{{\textbf{v}}} \end{array}} \right]^T$ where $\textbf{0}$ is the vector which has the same size as the vector $\textbf{v}$ but all its elements being zero. In what follows, we will show that 
there exists a vector $\textbf{v}$ such that $\textbf{u}^T\textbf{H}(\textbf{p}^*)\textbf{u}=\textbf{v}^T\textbf{H}_{22}(\textbf{p}^*)\textbf{v}<0$, which implies that $\textbf{H}(\textbf{p}^*)$ and $\textbf{H}_\textbf{V}(\textbf{p}^*)$ are not positive semidefinite. 
\begin{Lemma}
Let $\textbf{p}^*$ be an equilibrium in $\mathcal{Q}_I$. If there exists at least a vector $\textbf{v}$ such that $\textbf{v}^T\textbf{H}_{22}(\textbf{p}^*)\textbf{v}<0$ where $\textbf{H}_{22} (\textbf{p}^*)=2\textbf{R}_y^T(\textbf{p}^*){\textbf{R}_y}(\textbf{p}^*) + \textbf{E}(\textbf{p}^*)$, then $\textbf{p}^*$ is unstable.
\label{lemma:Hessian2D}
\end{Lemma}
\begin{figure}[htb]
\begin{center}
\includegraphics[width = 2 in]{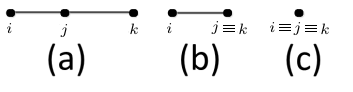}
\end{center}
\caption{Formations of three agents in undesired equilibria.}
\label{fig:incorrect2d}
\end{figure}
For the formation of triangle adding a flex node (Fig.\ref{fig:topologygraph}a), the undesired equilibrium set can be divided as ${\mathcal{Q}_I} = \mathcal{Q}_{I_1} \cup \mathcal{Q}_{I_2}$ where
\[\begin{array}{l}
\mathcal{Q}_{I_1} = \{ \textbf{p} \in \mathcal{Q}_I:{\textbf{z}_{34}} = 0\}, \\
\mathcal{Q}_{I_2} = \{ \textbf{p} \in \mathcal{Q}_I: \textrm{agents 1, 2, 3 are collinear}.\} 
\end{array}\]
Note that $\mathcal{Q}_{I_2}$ contains equilibria where agents are distinct and collinear (Fig. \ref{fig:incorrect2d}a), or there is a pair of agents that have the same position and a remaining agent that reaches desired distances from two others (Fig. \ref{fig:incorrect2d}b), or three agents are on the same position (Fig. \ref{fig:incorrect2d}c).
\begin{Lemma}
If $\textbf{p}^* \in \mathcal{Q}_{I_2}$, the formation of three agents $1,2,3$ has one of forms described in Fig. \ref{fig:incorrect2d}. The properties corresponding to each form are as follows
\begin{itemize}
\item Fig. \ref{fig:incorrect2d}a: $g_{ij}<0$, $g_{jk}<0$, $g_{ik}>0$ and $g_{ij}+g_{ik}<0$, $g_{jk}+g_{ik}<0$.
\item Fig. \ref{fig:incorrect2d}b: $g_{jk}<0$ and $g_{ij}=g_{ik}=0$.
\item Fig. \ref{fig:incorrect2d}c: $g_{ij}<0$, $g_{jk}<0$, $g_{ik}<0$.
\end{itemize}
\label{lemma:Incorrect2D}
\end{Lemma}
\begin{proof}
See Appendix A.
\end{proof}
\begin{Theorem}
For the system (\ref{eq:system}), whose underlying graph is a triangle adding a flex node moving in the plane, and the distance constraints set $D$ being feasible, the desired formation is almost globally asymptotically stable, i.e., $\textbf{p}^* \in \mathcal{Q}_I$ is unstable with respect to the control law (\ref{eq:gradient_controllaw_global}).
\label{theorem:Incorrect2D}
\end{Theorem}
\begin{proof}
The matrix $\textbf{H}_{22}(\textbf{p}^*)$ is a $4 \times 4$ symmetric matrix, $\textbf{H}_{22}(\textbf{p}^*)=[h_{ij}]_{4 \times 4}$, where $h_{ii} = \sum_{j\in \mathcal{N}_i}{\left( g_{ij}+\rho_{ij}(y_i-y_j)^2 \right)}$, $h_{ij}=-g_{ij}-\rho_{ij}(y_i-y_j)^2$ if $(i,j) \in \mathcal{E}$, $h_{ij}=0$ otherwise.
Consider undesired equilibrium $\textbf{p}^* \in \mathcal{Q}_I$.
\begin{itemize}
\item If $\textbf{p}^* \in \mathcal{Q}_{I_1}$, then $y_3=y_4, e_{34}=-\bar{d}_{34}^2<0$, which implies $g_{34}<0$. Consider the vector $\textbf{v}=\left[ {\begin{array}{*{20}{c}} {1}&{1}&{1}&{0} \end{array}} \right]^T$; then we have $\textbf{v}^T \textbf{H}_{22}(\textbf{p}^*) \textbf{v}=g_{34}<0$.
\item If $\textbf{p}^* \in \mathcal{Q}_{I_2}$, by Lemma \ref{lemma:Incorrect2D}, we have $g_{23}+g_{13}<0$. Without loss of generality, we choose the coordinate system such that agents 1, 2, 3 are on the $x$-axis. 
Consider the vector $\textbf{v}=\left[ {\begin{array}{*{20}{c}} {1}&{0}&{0}&{0} \end{array}} \right]^T$; then we have $\textbf{v}^T \textbf{H}_{22}(\textbf{p}^*) \textbf{v}=g_{23}+g_{23}<0$.
\end{itemize}
From the above analysis, the matrix $\textbf{H}_{22}(\textbf{p}^*)$ is not positive semidefinite for all undesired equilibrium ${\textbf{p}^*}\in \mathcal{Q}_I$; thus, every undesired equilibrium is unstable. Consequently, the desired formation is almost globally asymptotically stable.
\end{proof}

\subsection{Tetrahedron adding a flex node in the $3$-D space}
Similarly to the $2$-D case, we have 
\begin{Lemma} 
Let $\textbf{p}^*$ be an equilibrium in $\mathcal{Q}_I$; if there exists at least a vector $\textbf{v}$ such that $\textbf{v}^T\textbf{H}_{33}(\textbf{p}^*)\textbf{v}<0$ where $\textbf{H}_{33} (\textbf{p}^*)=2\textbf{R}_z^T(\textbf{p}^*){\textbf{R}_z}(\textbf{p}^*) + \textbf{E}(\textbf{p}^*)$, then $\textbf{p}^*$ is unstable.
\label{lemma:Hessian3D}
\end{Lemma}

For the formation of tetrahedron adding a flex node (Fig.\ref{fig:topologygraph}b), the undesired equilibria set can be divided as ${\mathcal{Q}_I} = \mathcal{Q}_{I_1} \cup \mathcal{Q}_{I_2}$ where
\[\begin{array}{l}
\mathcal{Q}_{I_1} = \{ \textbf{p} \in \mathcal{Q}_I:{\textbf{z}_{45}} = 0\} \\
\mathcal{Q}_{I_2} = \{ \textbf{p} \in \mathcal{Q}_I: \textrm{agents 1, 2, 3, 4 are coplanar}\} 
\end{array}\]
Here, $\mathcal{Q}_{I_2}$ contains equilibria where agents $1,2,3,4$ are in the same planar. The formation can be one of cases shown in Fig. \ref{fig:incorrect3d}.
\begin{figure}[htb]
\begin{center}
\includegraphics[width = 2.5 in]{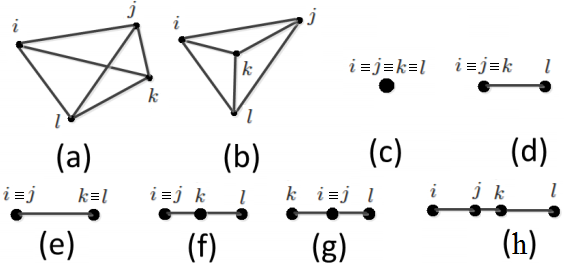}
\end{center}
\caption{Formation of four agents in undesired equilibria}
\label{fig:incorrect3d}
\end{figure}
To analyze the undesired equilibria in $\mathcal{Q}_{I_2}$, we employ the following lemmas:
\begin{Lemma}
(Lemma 5 in \cite{Summers2009}) Denote $\overline {{X}{Y}}$ be the length of edge from node X to node Y and $\angle XYZ$ be the angle between the vector YX and vector YZ. Consider the two triangles $A_1B_1C_1$ and $A_2B_2C_2$. If $\overline {{A_1}{B_1}}<\overline {{A_2}{B_2}}$, $\overline {{A_1}{C_1}}<\overline {{A_2}{C_2}}$, $\overline {{B_1}{C_1}}>\overline {{B_2}{C_2}}$, then $\angle B_1A_1C_1 > \angle B_2A_2C_2$.
\label{Lemma:angle_ineq}
\end{Lemma}
\begin{Lemma}
Let ABCD be a tetrahedral and angles at node A are $\theta_1,\theta_2,\theta_3$  as depicted in Fig. \ref{fig:tetrahedral}. Then $\theta_1+\theta_2+\theta_3<360^o$, $\theta_1+\theta_2>\theta_3,\theta_1+\theta_3>\theta_2$ and $\theta_2+\theta_3>\theta_1$.
\begin{figure}[htb]
\begin{center}
\includegraphics[width = 1.3 in]{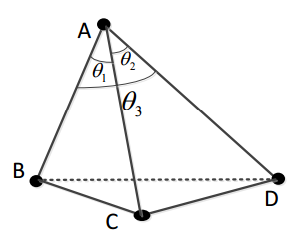}
\end{center}
\caption{A tetrahedron.}
\label{fig:tetrahedral}
\end{figure}
\label{Lemma:tetrahedron}
\end{Lemma}
Let the tetrahedron $A_1A_2A_3A_4$ correspond to the formation of four agents $1,2,3,4$ in $\mathcal{Q}_C$. We have:
\begin{equation}
\begin{array}{l}
\angle A_kA_iA_j + \angle A_lA_iA_k > \angle A_lA_iA_j,\\
\angle A_jA_kA_i + \angle A_iA_kA_l > \angle A_jA_kA_l
\end{array}
\label{eq:angleineq_1apply}
\end{equation}
\begin{equation}
\angle A_iA_kA_j + \angle A_lA_kA_i  + \angle A_jA_kA_l < 360^o, 
\label{eq:angleineq_2apply}
\end{equation}
and $\overline{A_iA_j}=\bar{d}_{ij}$ for $i,j,k,l \in \{1,2,3,4\}$ and $i,j,k,l$ are distinct.
\begin{Lemma}
If $\textbf{p}^* \in \mathcal{Q}_{I_2}$, the formation of four agents $1,2,3,4$ has one of forms as described in Fig. \ref{fig:incorrect3d}. The properties corresponding to each form are
\begin{itemize}
\item Fig. \ref{fig:incorrect3d}a: $g_{ij}<0, g_{jk}<0, g_{kl}<0, g_{il}<0, g_{ik}>0, g_{jl}>0$ and $g_{ij}+g_{ik}+g_{il}<0$, $g_{ij}+g_{jk}+g_{jl}<0$, $g_{ik}+g_{jk}+g_{kl}<0$, $g_{il}+g_{jl}+g_{kl}<0$.
\item Fig. \ref{fig:incorrect3d}b: $g_{ik}<0,g_{jk}<0,g_{kl}<0$ and $g_{ij}>0,g_{il}>0,g_{jl}>0$.
\item Fig. \ref{fig:incorrect3d}c: $g_{ij}<0$, $g_{ik}<0$, $g_{il}<0$, $g_{jk}<0$, $g_{jl}<0$, $g_{kl}<0$.
\item Fig. \ref{fig:incorrect3d}d: $g_{il}=g_{jl}=g_{kl}=0$, $g_{ij}<0$, $g_{jk}<0$, $g_{ik}<0$.
\item Fig. \ref{fig:incorrect3d}e: $g_{ik}+g_{il}=g_{jk}+g_{jl}=g_{ik}+g_{jk}=g_{il}+g_{jl}=0$, $g_{ij}<0$, $g_{kl}<0$.
\item Fig. \ref{fig:incorrect3d}f: $g_{il}+g_{jl}+g_{kl}<0$ and $g_{ij}+g_{ik}+g_{il}<0$ or $g_{ij}+g_{jk}+g_{jl}<0$.
\item Fig. \ref{fig:incorrect3d}g: $g_{il}+g_{jl}+g_{kl}<0$ and $g_{ik}+g_{jk}+g_{kl}<0$.
\item Fig. \ref{fig:incorrect3d}h: $g_{il}+g_{jl}+g_{kl}<0$ and $g_{ij}+g_{ik}+g_{il}<0$ if $g_{ij}<0$; $g_{ij}+g_{jk}+g_{jl}<0$ if $g_{ij}=0$; or $g_{ik}+g_{jk}+g_{kl}<0$ if $g_{ij}>0$.
\end{itemize}
\label{lemma:Incorrect3D}
\end{Lemma}
\begin{proof}
See Appendix B.
\end{proof}
\begin{Theorem}
For the given system (\ref{eq:system}), whose underlying graph is a tetrahedron adding a flex node moving in the $3$-D space, and the feasible distance constraints set $D$, the desired formation is almost globally asymptotically stable with respect to the control law (\ref{eq:gradient_controllaw_global}).
\label{theorem:Incorrect3D}
\end{Theorem}
\begin{proof}
Without loss of generality, we choose the coordinate system such that agents $1,2,3$ are in $x-y$ plane. The matrix $\textbf{H}_{33}(\textbf{p}^*)$ is a $5 \times 5$ symmetric matrix, $\textbf{H}_{33}(\textbf{p}^*)=[h_{ij}]_{5 \times 5}$, where $h_{ii} = \sum_{j\in \mathcal{N}_i}{\left( g_{ij}+\rho_{ij}(z_i-z_j)^2 \right)}$, $h_{ij}=-g_{ij}-\rho_{ij}(z_i-z_j)^2$ if $(i,j) \in \mathcal{E}$, and $h_{ij}=0$ otherwise.
\begin{itemize}
\item If $\textbf{p}^* \in \mathcal{Q}_{I_1}$, then $z_4=z_5, g_{45}<0$. Consider the vector $\textbf{v}=\left[ {\begin{array}{*{20}{c}} {1}&{1}&{1}&{1}&{0} \end{array}} \right]^T$; then we have $\textbf{v}^T \textbf{H}_{22}(\textbf{p}^*) \textbf{v}=g_{45}<0$.
\item If $\textbf{p}^* \in \mathcal{Q}_{I_2}$, then $g_{45}=0$ and $z_1=z_2=z_3=z_4=0$. We denote four agents of tetrahedral as $i,j,k,l$. If we omit the rotation and translation motions, at any undesired equilibrium, they will have one of the forms depicted in Fig.\ref{fig:incorrect3d}. From Lemma \ref{lemma:Incorrect3D}, at $\textbf{p}^* \in \mathcal{Q}_{I_2}$, there exist at least two agents $i,j$ such that $g_{ij}+g_{ik}+g_{il}<0$, $g_{ij}+g_{jk}+g_{jl}<0$.
Since agents $1,2,3$ have the same roles, we assume $g_{12}+g_{13}+g_{14}<0$. Consider the vector $\textbf{v}=\left[ {\begin{array}{*{20}{c}} {1}&{0}&{0}&{0}&{0} \end{array}} \right]^T$; then we have $\textbf{v}^T \textbf{H}_{33}(\textbf{p}^*) \textbf{v}=g_{12}+g_{13}+g_{14}<0$.
\end{itemize}
From the above analysis, the matrix $\textbf{H}_{33}(\textbf{p}^*)$ is not positive semidefinite for all undesired equilibrium ${\textbf{p}^*}\in \mathcal{Q}_I$; thus, every undesired equilibrium is unstable. Consequently, the desired formation is almost globally asymptotically stable.
\end{proof}
\section{Formation with additional flex agent moving as leader}
As discussed in the introduction, when a moving rigid formation has a flex node, the added flex node may act as a leader to guide the overall formation to a desired region. In this section, we assume the flex agent has an additional control input $\textbf{v}_f(t)$ satisfying one of the following two assumptions.
\begin{Assumption} 
$\textbf{v}_f(t)$ has the form as 
\[\textbf{v}_f(t)=\left\{ \begin{array}{lcl}
&\textbf{v}(t), \textrm{ } t \in [t_0, T_f], T_f < \infty \\
&0, \textrm{otherwise.}
\end{array} \right.\]
\label{as:flex_law1}
\end{Assumption}
\begin{Assumption}
The flex agent is required to go to a specific point $\textbf{p}_t$ and the additional input  has the form as $\textbf{v}_f(t) = k_f (\textbf{p}_t-\textbf{p}_{N+1})$ where $k_f \in \mathbb{R}^+$.
\label{as:flex_law2}
\end{Assumption} 

The dynamics of system can be written as
\begin{equation}
\dot {\textbf{p}} = \textbf{u} = - \nabla_{\textbf{p}} V(\textbf{p}) + \mathbf{\delta}_{N+1}\otimes\textbf{v}_f(t)
\label{eq:flex_controllaw_global}
\end{equation}
where $\mathbf{\delta} = [0,\ldots ,0, 1]^T   \in \mathbb{R}^{(N+1)}$.
\begin{Theorem}
For the system (\ref{eq:system}) with the interaction graph $\mathcal{G}$ which contains a rigid graph and a flex node addition and has the distance constraints set $D$ being feasible, under control law (\ref{eq:flex_controllaw_global}) where the additional control input $\textbf{v}_f(t)$ satisfies Assumption \ref{as:flex_law1} or Assumption \ref{as:flex_law2}, we have
\begin{itemize}
\item $\textbf{p}(t)$ approaches $\mathcal{Q}$ as $t \to \infty $.
\item The desired formation set $\mathcal{Q}_C$ is locally asymptotically stable.
\item If $\mathcal{G}$ is a triangle adding a flex node in the plane or a tetrahedron adding a flex node in the $3$-D space, $\mathcal{Q}_C$ is almost globally asymptotically stable.
\end{itemize}
\label{theorem:flex_effect}
\end{Theorem}
\begin{proof}
Denote $\textbf{z}_f = \textbf{p}_N - \textbf{p}_{N+1}$ and $g_f=g(||\textbf{z}_f||, \bar{d}_f)$. First, we will show that $\textbf{p}(t)$ approaches $\mathcal{Q}$ as $t \to \infty $ in both cases.
\begin{itemize}
\item Case of Assumption \ref{as:flex_law1}: Considering the potential function as (\ref{eq:potential_function_global}), we have:
$\dot{V}=- {\left\| {\frac{{\partial V}}{{\partial \textbf{p}}}} \right\|^2} - g_f\textbf{z}_f^T \textbf{v}_f (t)$
or
$\dot{V} + g_f \textbf{z}_f^T\textbf{v}_f(t)=- {\left\| {\frac{{\partial V}}{{\partial \textbf{p}}}} \right\|^2} \le 0$.
Then, the function $V(t) + \int_{t_0}^{\infty}{g_f \textbf{z}_f^T\textbf{v}_f(t)}$ is non-increasing and $\textbf{v}_f(t)$ is a decay function, i.e. $\int_{t_0}^{\infty}{g_f \textbf{z}_f^T\textbf{v}_f(t)}$ is bounded. So $V$ is bounded. According to Barbalat's lemma \cite[Lemma 8.2]{Khalil2002}, $\dot{V}$ will converge to $0$ or $\textbf{p}(t)$ will converge to $\mathcal{Q}$ as $t \to \infty $.
\item  Case of Assumption \ref{as:flex_law2}: Considering the potential function 
$V_f=V+k_f||\textbf{p}_t-\textbf{p}_{N+1}||^2$, we have
$\dot{V}_f=- \sum\limits_{i=1}^{N}{||\sum\limits_{j \in \mathcal{N}_i} {g_{ij} \textbf{z}_{ij}}||^2} - {|| g_f \textbf{z}_f+k_f (\textbf{p}_t-\textbf{p}_{N+1}) ||^2}.$
Since $\dot{V}_f \le 0$, $V_f$ is bounded and from Barbalat's lemma, $\dot{V}_f$ converges to zero as $t$ goes to infinity. Thus, $\dot{V}_1$ converges to zero, which means $\sum\limits_{j \in \mathcal{N}_i} {g_{ij} \textbf{z}_{ij}} \to 0$ for $i=1,...,N$ and $g_f \textbf{z}_f+k_f (\textbf{p}_t-\textbf{p}_{N+1}) \to 0$. From
\[\sum\limits_{i=1}^{N}{\left(\sum\limits_{j \in \mathcal{N}_i} g_{ij} \textbf{z}_{ij}\right)} = 0,\]
\[g_f \textbf{z}_f+k_f (\textbf{p}_t-\textbf{p}_{N+1}) = 0\]
by summing up the left-hand sides of the equations, we have $k_f (\textbf{p}_t-\textbf{p}_{N+1})=0$. So, $\textbf{p}(t)$ will converge to $\mathcal{Q}$ and $\textbf{p}_{N+1}$ will reach to the position $\textbf{p}_t$ as $t \to \infty $.
\end{itemize}
By the similar proof of Theorem 3.2 in \cite{Kwangkyo2014}, $\mathcal{Q}_C$ is locally asymptotically stable. The negative of derivative of the right-hand side of (\ref{eq:flex_controllaw_global}) is
$-\textbf{J} = \textbf{H}_\textbf{V} + \left[\delta_{N+1}^T\delta_{N+1}\right] \otimes \frac{\partial \textbf{v}_f}{\partial \textbf{p}_{N+1}}$.
When the underlying graph is a triangle adding a flex node in $2$-D (or a tetrahedron adding a flex node in $3$-D), by the similar analysis as the proof of Theorem \ref{theorem:Incorrect2D} (resp. Theorem \ref{theorem:Incorrect3D}), we see that for any undesired equilibrium $\textbf{p}^* \in \mathcal{Q}_I$, there exists a vector $\textbf{v}$ such that $\textbf{v}^T\textbf{H}_\textbf{V}(\textbf{p}^*)\textbf{v}<0$ and the element corresponding to flex agent in the vector $\textbf{v}$ is zero; or $\textbf{v}^T(-\textbf{J}_f)(\textbf{p}^*)\textbf{v}<0$. It implies all undesired equilibria are unstable or the desired formation set is almost globally asymptotically stable.
\end{proof}
\section{Simulations}
In this section, we conduct simulation of formations under the control law (\ref{eq:gradient_controllaw_global}) where $\phi(e_{\epsilon_k},\bar{d}_{\epsilon_k})=\frac{1}{2}e_{\epsilon_k}^2$ and the control law (\ref{eq:flex_controllaw_global}) where $\textbf{v}_f(t)=5(\textbf{p}_t-\textbf{p}_{N+1})$.
\subsection{A triangle adding a flex node in the plane}
We consider the case with desired distances $\bar{d}_{12}=\bar{d}_{23}=\bar{d}_{13}=\bar{d}_{34}=4$ and agents with initial positions $\textbf{p}_1(0) = \left[ {12;2} \right]$, $\textbf{p}_2(0) = \left[ {-12;2} \right]$, $\textbf{p}_3(0) = \left[ {0;-2} \right]$, $\textbf{p}_4(0) = \left[ {0;9.228} \right]$. Simulation results are shown in Fig. \ref{fig:simulation1} (corresponding to control law (\ref{eq:gradient_controllaw_global})) and Fig. \ref{fig:simulation2} (corresponding to control law (\ref{eq:flex_controllaw_global})).
\begin{figure}
\begin{center}
\includegraphics[width = 2.5 in]{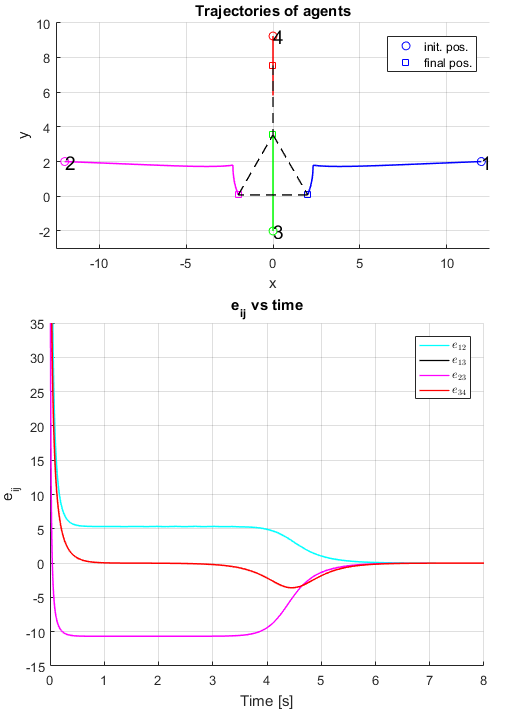}
\end{center}
\caption{At $t=0.3$s, agents 1, 2, 3 are collinear and $\dot{\textbf{p}}_1=\dot{\textbf{p}}_2=\dot{\textbf{p}}_3=\dot{\textbf{p}}_4=0$ and distance between agents 4 and 3 is $\bar{d}_{34}$; at $t=3.1$, agent 4 has a small movement, then the system converges to the desired formation.}
\label{fig:simulation1}
\end{figure}
\begin{figure}
\begin{center}
\includegraphics[width = 2.5 in]{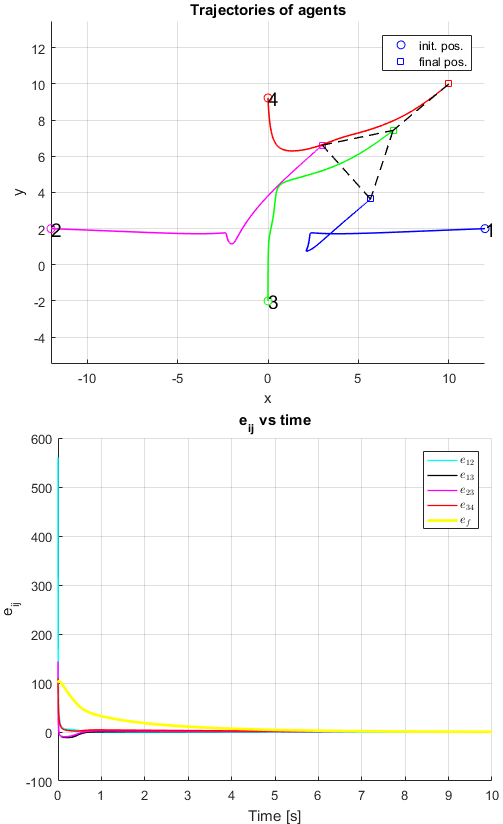}
\end{center}
\caption{The flex agent $4$ converges to the target point $\textbf{p}_t = \left[ {10;10} \right]$ and the square distance error converges to zero}
\label{fig:simulation2}
\end{figure}
\subsection{A tetrahedral adding a flex node in $3$-D space}
We consider the case with desired distances $\bar{d}_{12}=\bar{d}_{23}=\bar{d}_{13}=\bar{d}_{14}=\bar{d}_{24}=\bar{d}_{34}=\bar{d}_{45}=4$ and agents with initial positions $\textbf{p}_1(0) = \left[ \frac{20}{\sqrt{3}};0;0 \right]$, $\textbf{p}_2(0) = \left[ \frac{-10}{\sqrt{3}};10;0 \right]$, $\textbf{p}_3(0) = \left[ \frac{-10}{\sqrt{3}};-10;0 \right]$, $\textbf{p}_4(0) = \left[ 0;0;-3 \right]$, $\textbf{p}_5(0) = \left[ 0;0;6.485 \right]$. Simulation results are shown in Fig. \ref{fig:simulation3} (corresponding to control law (\ref{eq:gradient_controllaw_global})) and Fig. \ref{fig:simulation4} (corresponding to control law (\ref{eq:flex_controllaw_global})). 
\begin{figure}
\begin{center}
\includegraphics[width = 2.5 in]{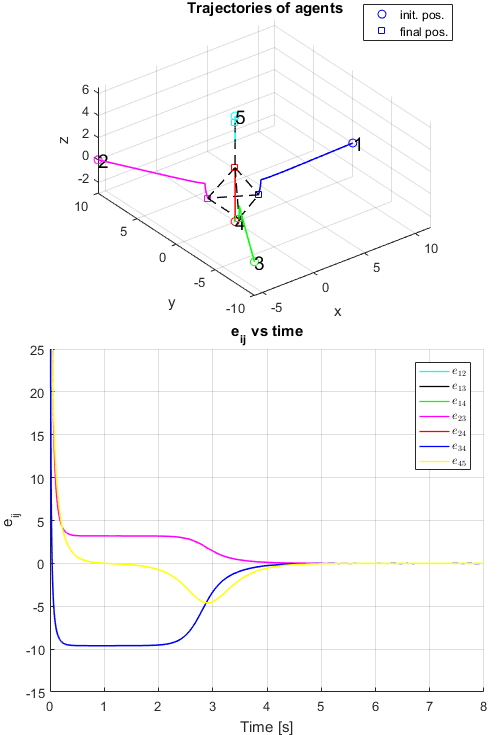}
\end{center}
\caption{At $t=0.8$s, agents 1, 2, 3, 4 are co-planar and $\dot{\textbf{p}}_1=\dot{\textbf{p}}_2=\dot{\textbf{p}}_3=\dot{\textbf{p}}_4=\dot{\textbf{p}}_5=0$ and distance between agents 4 and 5 is $\bar{d}_{45}$; at $t=1.5$, agent 5 has a small movement, then the system converges to the desired formation.}
\label{fig:simulation3}
\end{figure}
\begin{figure}
\begin{center}
\includegraphics[width = 2.5 in]{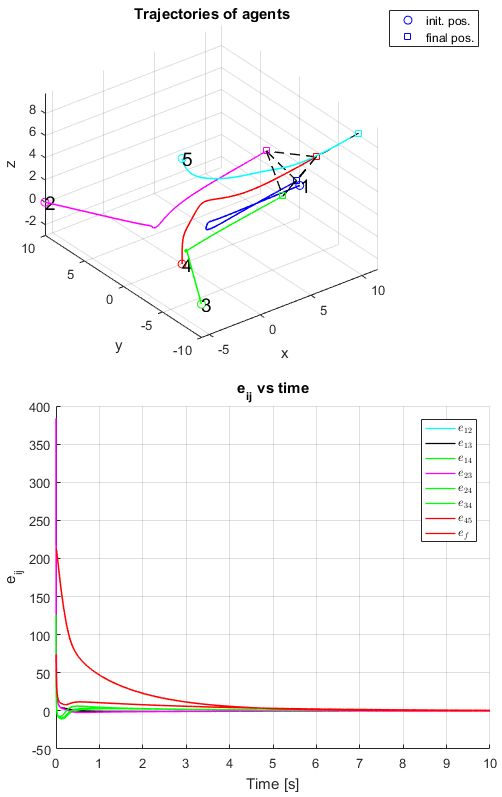}
\end{center}
\caption{The flex agent $5$ converges to target point $\textbf{p}_t = \left[ {10;-10;10} \right]$ and the square distance error converges to zero}
\label{fig:simulation4}
\end{figure}
\section{Conclusion}
In this paper, we studied the distance-based formation control of a group of autonomous agents, whose underlying graph consists of a rigid graph and a flex node addition. Under the gradient control law, the desired formation, where all distance constraints between neighboring agents are achieved, is locally asymptotically stable. We examined stability of undesired equilibrium points with two specific configurations: a triangle adding a flex node in the plane and a tetrahedron adding a flex node in the $3$-D space. We showed that the Hessian of potential function is not positive semi-definite at any undesired equilibrium point; thus, we can conclude that all undesired equilibria are unstable, which means that the desired formation set is almost globally asymptotically stable.
\bibliographystyle{IEEEtran}
\bibliography{mylib}
\appendices
\section{Proof of Lemma \ref{lemma:Incorrect2D}}
\begin{itemize}
\item Case of Fig.\ref{fig:incorrect2d}a: In this case, $||{\textbf{z}}_{ij}|| + ||{\textbf{z}}_{jk}|| = ||{\textbf{z}}_{ik}||$. 
Let us assume ${g}_{ik} < 0$. With the consideration of the balances of agents $i$ and $k$, we obtain ${g}_{ij} > 0$, ${g}_{jk} > 0$. These three signs imply that $\left\| {{{\textbf{z}}_{ik}}} \right\| < \bar{d}_{ik},\left\| {{{\textbf{z}}_{ij}}} \right\| > \bar{d}_{ij},\left\| {{{\textbf{z}}_{jk}}} \right\| > \bar{d}_{jk}$. Then we have $\bar{d}_{ik} > \bar{d}_{ij} + \bar{d}_{jk}$, which contradicts (\ref{eq:relizable_ineq}). Similarly, if ${e}_{ik} = 0$, it implies $\bar{d}_{ik} = \bar{d}_{ij} + \bar{d}_{jk}$, which contradicts (\ref{eq:relizable_ineq}) too.
So, ${g}_{ik} > 0$, ${g}_{ij} < 0$, and ${g}_{jk} < 0$. Since $g_{ij}\textbf{z}_{ij}+g_{ik}\textbf{z}_{ik} = 0$ and $0<||\textbf{z}_{ij}||<||\textbf{z}_{ik}|||$, we have $g_{ij}+g_{ik}=\left(1-\frac{||\textbf{z}_{ij}||}{||\textbf{z}_{ik}||}\right)g_{ij}=\left(1-\frac{||\textbf{z}_{ik}||}{||\textbf{z}_{ij}||}\right)g_{ik}<0$. Similarly, $g_{jk}+g_{ik}<0$.
\item Case of Fig.\ref{fig:incorrect2d}b: In this case, we have $g_{jk}\textbf{z}_{jk}+g_{ik}\textbf{z}_{ik} = 0$ and $\textbf{z}_{jk}=\textbf{0}, ||\textbf{z}_{ik}|| \neq 0$; so $g_{ik}=0$. Similarly, $g_{ij}=0$. Since $e_{jk}=-\bar{d}_{jk}^2<0$, we have $g_{jk}=g(e_{jk},\bar{d}_{jk})<0$.
\item Case of Fig.\ref{fig:incorrect2d}c: From $\textbf{z}_{ij}=\textbf{z}_{jk}=\textbf{z}_{ik}=\textbf{0}$, we have $e_{ij}<0, e_{jk}<0, e_{ik}<0$ or $g_{jk}<0$, $g_{jk}<0$, $g_{ik}<0$.
\end{itemize}
\section{Proof of Lemma \ref{lemma:Incorrect3D}}
\begin{itemize}
\item Case of Fig.\ref{fig:incorrect3d}a: From the balance of agent $i$, we have
$g_{ij}\textbf{z}_{ij}+g_{ik}\textbf{z}_{ik}+g_{il}\textbf{z}_{il}=0$.
Let $\textbf{p}^*_\lambda$ be the intersected point of the line containing agents $i,k$ and the line containing agents $j,l$. Then, we can calculate the coordinate of $\textbf{p}^*_\lambda$ as
$\textbf{p}^*_\lambda=\lambda \textbf{p}^*_j + (1-\lambda)\textbf{p}^*_l, \textrm{  } \lambda \in (0,1)$
and we have $\textbf{z}_{ik}=\alpha(\textbf{p}^*_i-\textbf{p}^*_\lambda)$ for some $\alpha>1$.
Combining above two equations with the balance equation of agent $i$, we can obtain
\[(\alpha \lambda g_{ik} + g_{ij}) \textbf{z}_{ij} + (\alpha(1-\lambda)g_{ik}+g_{il})\textbf{z}_{il}=0\]
Thus, we have 
\begin{subequations}
\begin{align}
\alpha \lambda g_{ik}+g_{ij}=0\\
\alpha(1-\lambda)g_{ik}+g_{il}=0
\end{align}
\label{eq:sign}
\end{subequations}
From (\ref{eq:sign}a), we have $sgn(g_{ij})=-sgn(g_{ik})$, and from (\ref{eq:sign}b), we have $sgn(g_{il})=-sgn(g_{ik})$. Similarly, we can obtain $sgn(g_{ij})=-sgn(g_{jl})$, $sgn(g_{jk})=-sgn(g_{jl})$, $sgn(g_{jk})=-sgn(g_{ik})$, $sgn(g_{kl})=-sgn(g_{ik})$, $sgn(g_{il})=-sgn(g_{jl})$, $sgn(g_{kl})=-sgn(g_{jl})$. Thus, $sgn(g_{ij})=sgn(g_{jk})=sgn(g_{kl})=sgn(g_{il})=-sgn(g_{ik})=-sgn(g_{jl})$. Assume that $g_{ik}<0$, then $g_{jl}<0$, $g_{ij}>0$, $g_{jk}>0$, $g_{kl}>0$, $g_{il}>0$. Consequently, $\left\| {{\textbf{z}}_{ij}} \right\| < {\bar{d}_{ij}},\left\| {{\textbf{z}}_{jk}} \right\| > {d_{jk}},\left\| {{\textbf{z}}_{kl}} \right\| > {\bar{d}_{kl}},\left\| {{\textbf{z}}_{il}} \right\| > {\bar{d}_{il}},\left\| {{\textbf{z}}_{ik}} \right\| < {\bar{d}_{ik}},\left\| {{\textbf{z}}_{jl}} \right\| < {\bar{d}_{jl}}$. 
Apply Lemma \ref{Lemma:angle_ineq} to four pairs of triangles $A_iA_jA_k$ and $ijk$, $A_jA_kA_l$ and $jkl$, $A_kA_lA_i$ and $kli$, $A_lA_iA_j$ and $lij$, then we have:
\begin{equation}
\begin{array}{l}
\angle A_iA_jA_k > \angle ijk, \angle A_jA_kA_l > \angle jkl\\
\angle A_kA_lA_i > \angle kli, \angle A_lA_iA_j > \angle lij
\end{array}
\label{eq:angle_ineq_apply}
\end{equation} 
Since $\angle A_iA_jA_k + \angle A_jA_kA_i + \angle A_k{\rm{A_iA_j}} = \angle {i}{j}{k} + \angle {j}{k}{i} + \angle {k}{i}{j} = {180^o}$ but $\angle A_iA_jA_k > \angle ijk$, we obtain
\begin{equation}
\angle A_jA_kA_i + \angle A_k{\rm{A_iA_j}} < \angle {j}{k}{i} + \angle {k}{i}{j}
\label{eq:angle_ineq_apply_1}
\end{equation}
 Similarly, we have 
 \begin{equation}
\angle A_lA_iA_k + \angle A_l{\rm{A_kA_i}} < \angle {l}{i}{k} + \angle {l}{k}{i}
\label{eq:angle_ineq_apply_2}
\end{equation}
Combining (\ref{eq:angle_ineq_apply_1}) and (\ref{eq:angle_ineq_apply_2}), we have 
$\angle A_jA_kA_i + \angle A_k{\rm{A_iA_j}} + \angle A_lA_iA_k + \angle A_l{\rm{A_kA_i}} < \angle {j}{k}{i} + \angle {k}{i}{j} + \angle {l}{i}{k} + \angle {l}{k}{i}$
or $ \angle A_jA_kA_i + \angle A_k{\rm{A_iA_j}} + \angle A_lA_iA_k + \angle A_l{\rm{A_kA_i}} < \angle {j}{k}{l} + \angle {l}{i}{j}$,
which implies that from (\ref{eq:angle_ineq_apply}) we have $\angle A_jA_kA_i + \angle A_k{\rm{A_iA_j}} + \angle A_lA_iA_k + \angle A_l{\rm{A_kA_i}} < \angle A_jA_kA_l + \angle A_lA_iA_j$. This contradicts (\ref{eq:angleineq_1apply}). Assume $g_{ik}=0$, by similar analysis, we get $\angle A_jA_kA_i + \angle A_k{\rm{A_iA_j}} + \angle A_lA_iA_k + \angle A_l{\rm{A_kA_i}} = \angle A_jA_kA_l + \angle A_lA_iA_j$. This contradicts (\ref{eq:angleineq_1apply}), too. So, $g_{ik}>0$, and it implies $g_{jl}>0$, $g_{ij}<0$, $g_{jk}<0$, $g_{kl}<0$, $g_{li}<0$. 
Let $j',l'$ be the projection of $j,l$ onto edge $(i,j)$, since $ijkl$ is a convex quadrilateral, then $||\textbf{z}_{ij'}||<||\textbf{z}_{ik}||, ||\textbf{z}_{il'}||<||\textbf{z}_{ik}||$. From the balance of agent $i$, we have $g_{ik}||\textbf{z}_{ik}||+g_{il}||\textbf{z}_{il'}||+g_{ij'}||\textbf{z}_{ij'}||=0$, then  $g_{ij}+g_{ik}+g_{il}=\left(1-\frac{||\textbf{z}_{ij'}||}{||\textbf{z}_{ik}||}\right)g_{ij}+\left(1-\frac{||\textbf{z}_{il'}||}{||\textbf{z}_{ik}||}\right)g_{il}<0$. Similarly, we have $g_{ij}+g_{jk}+g_{jl}<0$, $g_{ik}+g_{jk}+g_{kl}<0$, $g_{il}+g_{jl}+g_{kl}<0$.
\item Case of Fig.\ref{fig:incorrect3d}b: From the balance of agent $i$, by following the similar process as above, we have the same equation as (\ref{eq:sign}) with $0<\alpha<1$, and $sgn(g_{ij})=sgn(g_{il})=-sgn(g_{ik})$.
We can obtain $sgn(g_{ij})=sgn(g_{jl})=-sgn(g_{jk})$ and $sgn(g_{jl})=sgn(g_{il})=-sgn(g_{kl})$. Thus, $sgn(g_{ij})=sgn(g_{jl})=sgn(g_{il})=-sgn(g_{ik})=-sgn(g_{jk})=-sgn(g_{kl})$.
Assume that $g_{ik}>0$, then we have $g_{jk}>0$, $g_{kl}>0$, $g_{ij}<0$, $g_{jl}<0$, $g_{il}<0$. Consequently, $\left\| {{\textbf{z}}_{ik}} \right\| > {\bar{d}_{ik}},\left\| {\textbf{z}_{jk}} \right\| > {\bar{d}_{jk}},\left\| {\textbf{z}_{kl}} \right\| > {\bar{d}_{kl}},\left\| {\textbf{z}_{ij}} \right\| < {\bar{d}_{ij}},\left\| {\textbf{z}_{jl}} \right\| < {\bar{d}_{jl}},\left\| {\textbf{z}_{il}} \right\| < {\bar{d}_{il}}$. 
Apply Lemma \ref{Lemma:angle_ineq} to three pairs of triangles $A_iA_kA_j$ and $ikj$, $A_jA_kA_l$ and $jkl$, $A_lA_kA_i$ and $lki$, then we have
$\angle A_iA_kA_j > \angle ikj,
\angle A_jA_kA_l > \angle jkl,
\angle A_lA_kA_i > \angle lki$
or $\angle A_iA_kA_j + \angle A_jA_kA_l + \angle A_lA_kA_i > \angle {i}{k}{j} + \angle {j}{k}{l} + \angle {l}{k}{j} = {360^o}$, which contradicts (\ref{eq:angleineq_2apply}). Similarly with assumption $g_{ik}=0$, we can obtain $\angle A_iA_kA_j + \angle A_jA_kA_l + \angle A_lA_kA_i = \angle {i}{k}{j} + \angle {j}{k}{l} + \angle {l}{k}{j} = {360^o}$, which contradicts (\ref{eq:angleineq_2apply}), too. So, $g_{ik}<0$, $g_{jk}<0$, $g_{kl}<0$, $g_{ij}>0$, $g_{jl}>0$, $g_{il}>0$.
\item Case of Fig.\ref{fig:incorrect3d}c: Four agents $i,j,k,l$ have the same position, we have $e_{ij}=-\bar{d}^2_{ij}<0$ which implies $g_{ij}<0$. Similarly, $g_{ik}<0,g_{il}<0,g_{jk}<0,g_{jl}<0,g_{il}<0$.
\item Case of Fig.\ref{fig:incorrect3d}d: Three agents $i,j,k$ have the same position and $||\textbf{z}_{il}||=||\textbf{z}_{jl}||=||\textbf{z}_{kl}||>0$. From the balance of agents $i,j,k$, we have $g_{il}=g_{jl}=g_{kl}=0$. Since $e_{ij}=-\bar{d}^2_{ij}<0$, we have $g_{ij}<0$. Similarly, $g_{ik}<0,g_{jk}<0$.
\item Case of Fig.\ref{fig:incorrect3d}e: Two pairs of neighboring agents have the same position. Let $\textbf{p}^*_i=\textbf{p}^*_j$ and $\textbf{p}^*_k=\textbf{p}^*_l$. From the balance of agent $i$, we have $g_{ik}(\textbf{p}^*_i-\textbf{p}^*_k)+g_{il}(\textbf{p}^*_i-\textbf{p}^*_l)=(g_{ik}+g_{il})(\textbf{p}^*_i-\textbf{p}^*_l)=0$. Then, it follows $g_{ik}+g_{il}=0$ and $e_{ij}=-d^2_{ij}<0$ which implies $g_{ij}<0$. Similarly, $g_{jk}+g_{jl}=0$, $g_{jk}+g_{ik}=0$, $g_{il}+g_{jl}=0$, and $g_{kl}<0$.

\ \

To consider the cases of Figs. \ref{fig:incorrect3d}f, g, h, from the similar analysis as the proof of Lemma \ref{lemma:Incorrect2D}, we use
\begin{itemize}
\item[--]  For the neighboring pairs $(a,b), (a,c), (b,c)$, suppose that $a,b,c$ are collinear and $||\textbf{z}_{ab}||<||\textbf{z}_{ac}||$. If the distance constraints set is feasible, then $g_{ac} \le 0$ and $g_{ab} \ge 0$ imply $g_{bc}<0$.
\item[--] Suppose that agents $a,b,c,d$ are collinear, and $||\textbf{z}_{ab}||<||\textbf{z}_{ac}||<||\textbf{z}_{ad}||$ and $g_{ab}\textbf{z}_{ab}+g_{ac}\textbf{z}_{ac}+g_{ad}\textbf{z}_{ad}=0$. If $g_{ab} < 0$ or $g_{ad} > 0$, then $g_{ab}+g_{ac}+g_{ad}<0$.
\end{itemize}

Now we consider the stability of $\textbf{p}^* \in \mathcal{Q}_I$.
\item Case of Fig.\ref{fig:incorrect3d}f: From the balance of agents $i,j$, we have ${\mathop{\rm sgn}} (g_{ik}) =  - {\mathop{\rm sgn}} (g_{il})$, $|g_{ik}|>|g_{il}|$, ${\mathop{\rm sgn}} (g_{jk}) =  - {\mathop{\rm sgn}} (g_{jl})$, and $|g_{jk}|>|g_{jl}|$.
Suppose that $g_{ik}>0$ and $g_{jk}>0$. From the balance of agent $k$, we have $g_{kl}>0$. Since $g_{ik}>0,g_{kl}>0$ and $g_{il}<0$, it means that $\bar{d}_{il}>||\textbf{z}_{il}||=||\textbf{z}_{ik}||+||\textbf{z}_{kl}||>\bar{d}_ik+\bar{d}_kl$. This contradicts (\ref{eq:relizable_ineq}); so $g_{ik} \le 0$ or $g_{jk} \le 0$.
\begin{itemize}
\item[--] If $g_{ik} \le 0$: From the balance of agent $i$, we have $g_{il} \ge 0$, $g_{il} \ge 0$ and $g_{ik} \le 0$, which implies $g_{kl}<0$. Since $||\textbf{z}_{ik}||<||\textbf{z}_{il}||$, we have $g_{ij}+g_{ik}+g_{il} \le g_{ij}<0$ and $g_{il}+g_{jl}+g_{kl} <0$.
\item[--] If $g_{jk} \le 0$: Similarly, we have $g_{ij}+g_{jk}+g_{jl} \le g_{ij}<0$ and $g_{il}+g_{jl}+g_{kl} <0$.
\end{itemize}
\item Case of Fig.\ref{fig:incorrect3d}g: From the balance of agents $i$ and $j$, we have ${\mathop{\rm sgn}} (g_{ik}) =  {\mathop{\rm sgn}} (g_{il})$ and ${\mathop{\rm sgn}} (g_{jk}) =  {\mathop{\rm sgn}} (g_{jl})$. Assume $g_{ik} \ge 0$ and $g_{jk} \ge 0$, then $g_{il}>0, g_{jl}>0$. So $g_{kl}>0$; but this cannot happen due to the balance of agent $k$ and distance constraints set being feasible.
Observe that $i$ and $j$ have the same role in this case. Consider $g_{ik} < 0$. If $g_{kl}<0$ then $g_{jk}>0$, which implies $g_{jl}>0$, but this cannot happen. So, $g_{kl}>0$. Since $||\textbf{z}_{kl}||>||\textbf{z}_{ki}||, ||\textbf{z}_{kl}||>||\textbf{z}_{kj}||, ||\textbf{z}_{lk}||>||\textbf{z}_{li}||, ||\textbf{z}_{lk}||>||\textbf{z}_{li}||$, we have $g_{ik}+g_{jk}+g_{kl} < 0$ and $g_{il}+g_{jl}+g_{kl} < 0$.
\item Case of Fig.\ref{fig:incorrect3d}h: There are some available possibilities:
\begin{itemize}
\item[--] $g_{ij}<0,g_{il} > 0$: Since  ${||\textbf{z}_{ij}||<||\textbf{z}_{ik}||<||\textbf{z}_{il}||}$ and ${||\textbf{z}_{lk}||<||\textbf{z}_{lj}||<||\textbf{z}_{li}||}$, we have $g_{ij}+g_{ik}+g_{il} < 0$, $g_{il}+g_{jl}+g_{kl} < 0$.
\item[--] $g_{ij}<0, g_{il} \le 0$: From the balance of agent $i$, we have $g_{ik}>0$. Thus $g_{ik}>0$ and $g_{il} \le 0$ imply $g_{kl}<0$. Then, $g_{ij}+g_{ik}+g_{il} < 0$ and $g_{il}+g_{jl}+g_{kl} < 0$.
\item[--] $g_{ij}=0,g_{il} \ge 0$: If $g_{il} > 0$, from the balance of agent $i$, we have $g_{ik} < 0$, which implies $g_{jk}<0$. From the balance of agent $j$, we have $g_{jl}>0$. Also, from $||\textbf{z}_{jk}||<||\textbf{z}_{jl}||, {||\textbf{z}_{lk}||<||\textbf{z}_{lj}||<||\textbf{z}_{li}||}$, we have $g_{ij}+g_{jk}+g_{jl} < 0$ and $g_{il}+g_{jl}+g_{kl} < 0$. If $g_{ij}=g_{ik}=g_{il}=0$, by the similar proof as in Lemma \ref{lemma:Incorrect2D}, we have the same result.
\item[--] $g_{ij}\ge 0, g_{il} < 0$: From the balance of agent $i$, we have $g_{ik}>0$. Thus, $g_{ij}\ge 0$ and $g_{il} < 0$ imply $g_{jl}<0$. Also $g_{il}< 0$ and $g_{ik} > 0$ imply $g_{kl}<0$. But, this cannot happen due to the balance of agent $l$.
\item[--] $g_{ij}>0,g_{il} \ge 0$: From the balance of agent $i$, we have $g_{ik}<0$. Thus, $g_{ik}< 0$ and $g_{ij} > 0$ imply $g_{jk}<0$. From the balance of agent $k$, we have $g_{kl}<0$. So $g_{ik}+g_{jk}+g_{kl} < 0$. Since ${||\textbf{z}_{lk}||<||\textbf{z}_{lj}||<||\textbf{z}_{li}||}$, we have $g_{il}+g_{jl}+g_{kl} < 0$.
\end{itemize}
\end{itemize}
\end{document}